\documentclass[11pt,reqno]{amsart}

\usepackage{amssymb, latexsym, amsmath}
\usepackage{amsfonts}
\usepackage{amsthm}
\usepackage{enumerate}
\usepackage{setspace}
\usepackage{hyperref}

\usepackage{graphicx}
\usepackage{palatino}
\usepackage[T1]{fontenc}
\usepackage{setspace}

\usepackage{wrapfig}

\usepackage{tikz}

\usetikzlibrary{calc}
\usetikzlibrary{positioning}
\usepackage{tkz-graph}
\usetikzlibrary{decorations.markings}
\usetikzlibrary{decorations.pathreplacing}
\usetikzlibrary{patterns}

\theoremstyle{definition}
\newtheorem{theorem}{Theorem}
\newtheorem{lemma}[theorem]{Lemma}
\newtheorem{proposition}[theorem]{Proposition}

\newtheorem{corollary}[theorem]{Corollary}
\newtheorem{fact}[theorem]{Fact}

\title[Minimum degree conditions for small percolating sets]{Minimum degree conditions for small percolating sets in bootstrap percolation}

\author[K. Gunderson]{Karen Gunderson}
\address{Department of Mathematics, University of Manitoba, Winnipeg MB R3T\thinspace2N2, Canada}
\email{karen.gunderson@umanitoba.ca}
\thanks{Research supported in part by NSERC grant RGPIN-2016-05949}

\date{30 March 2017}

\subjclass[2010]{60K35, 05C35}
\keywords{Bootstrap percolation, minimum percolating sets}

\begin{document}

\begin{abstract}
The $r$-neighbour bootstrap process is an update rule for the states of vertices in which `uninfected' vertices with at least $r$ `infected' neighbours become infected and a set of initially infected vertices is said to \emph{percolate} if eventually all vertices are infected.  For every $r \geq 3$, a sharp condition is given for the minimum degree of a sufficiently large graph that guarantees the existence of a percolating set of size $r$.  In the case $r=3$, for $n$ large enough, any graph on $n$ vertices with minimum degree $\lfloor n/2 \rfloor +1$ has a percolating set of size $3$ and for $r \geq 4$ and $n$ large enough (in terms of $r$), every graph on $n$ vertices with minimum degree $\lfloor n/2 \rfloor + (r-3)$ has a percolating set of size $r$.  A class of examples are given to show the sharpness of these results.
\end{abstract}

\maketitle

\section{Introduction}\label{sec:intro}

Bootstrap percolation is a model for the spread of an `infection' in a network.  The $r$-neighbour bootstrap processes are an example of a cellular automaton.  The notion of cellular automata were introduced by von Neumann~\cite{jvN66} after a suggestion of Ulam~\cite{sU52}.  In this paper, an extremal problem related to these processes is considered.

For any integer $r \geq 2$, the $r$-neighbour bootstrap process is an update rule for the states of vertices in a graph which are in one of two possible states at any given time: `infected' or `uninfected'.  From an initial configuration of infected and uninfected vertices, the following state updates occur simultaneously and at discrete time steps: any uninfected vertex with at least $r$ infected neighbours becomes infected while infected vertices remain infected forever.  To be precise, given a graph $G$ and a set $A \subseteq V(G)$ of `initially infected' vertices, set $A_0 = A$ and for every $t \geq 1$ define
\[
A_t = A_{t-1} \cup \{v \in V(G) \mid |N(v) \cap A_{t-1}| \geq r\}.
\]
The \emph{closure} of $A$ is $\langle A \rangle_r = \cup_{t \geq 0} A_t$; the set of vertices that are eventually infected starting from $A = A_0$.  The set $A_t \setminus A_{t-1}$ shall often be referred to as the vertices \emph{infected at time step $t$}.  The set $A$ is said to \emph{span} $\langle A \rangle_r$.  The set $A$ is called \emph{closed} if{f} $\langle A \rangle_r = A$ and is said to \emph{percolate} if{f} $\langle A \rangle_r = V(G)$.  
  
The class of $r$-neighbour bootstrap processes were first introduced and investigated by Chalupa, Leath, and Reich~\cite{CLR79} as a monotone model of the dynamics of ferromagnetism.

While the focus of study for such processes is often the behaviour of initially infected sets that are chosen at random, a number of natural extremal problems arise.  For any graph $G$ and $r \geq 2$, define the size of the smallest percolating set to be
\[
m(G, r) = \min\{|A| \mid A \subseteq V(G),\ \langle A \rangle_r = V(G)\}.
\]

One class of graphs that have received a great deal of attention in this area are the square grids.  For any $n$ and $d$, let $[n]^d$ denote the $d$-dimensional $n \times n \times \cdots \times n$ grid.  In the case that $r = 2$, for all $n$ and $d$, the quantity $m([n]^d, 2)$ is known exactly (see~\cite{BB06} and \cite{BBM10}).  Pete (see~\cite{BP98}) gave a number of general results about the smallest percolating sets in grids with other thresholds and observed that $m([n]^d, d) = n^{d-1}$.  In the case of hypercubes, $Q_d = [2]^d$, Morrison and Noel~\cite{MN15} confirmed a conjecture of Balogh and Bollob\'{a}s~\cite{BB06}, showing that for each fixed $r$, $m(Q_d, r) = \frac{(1+o(1))}{r} \binom{d}{r-1}$. 

 Minimum percolating sets in trees were investigated by Riedl~\cite{eR10}.

The size of minimum percolating sets in regular graphs have been examined by Coja-Oghlan, Feige, Krivelevich and Reichman~\cite{C-ORKR} who gave bounds on $m(G, r)$ in a number of different cases in which $G$ is a regular graph satisfying various expansion properties.  Bounds on the size of a minimum percolating set (or `contagious set') in both binomial random graphs and random regular graphs have been given by Feige, Krivelevich, and Reichman~\cite{FKR16} and Guggiola and Semerjian~\cite{GS15}.

 Extremal problems for more general `$\mathcal{H}$-bootstrap processes' were considered by Balogh, Bollob\'{a}s, Morris, and Riordan~\cite{BBMR12} and many other natural extremal problems have been examined including the largest minimal percolating sets~\cite{rM09} and the `percolation time'~\cite{BP13, BP15, mP12}.

In this note, we shall focus on the conditions for the minimum degree of a graph that imply the existence of a percolating set of the smallest possible size.  It is clear that for any graph on at least $r$ vertices, $m(G, r) \geq r$.  Throughout, $\delta(G)$ is used to denoted the minimum degree of a graph $G$.

Considering the degree sequence of a graph, Reichman~\cite{dR12} showed that for any any graph $G$ and threshold $r$, then 
\[
m(G, r) \leq \sum_{v \in V(G)} \min\left\{1, \frac{r}{\deg(v) + 1} \right\}.
\]
  For any $d \geq r-1$, this upper bound is achieved by disjoint copies of cliques on $d+1$ vertices.

Freund, Poloczek, and Reichman~\cite{FPR15} showed that if $G$ is a graph on $n$ vertices with $\delta(G) \geq \left\lceil \frac{(r-1)}{r}n \right\rceil$, then $m(G, r) = r$.  Furthermore, they gave the example for odd $r$ of a clique on $n = r+1$ vertices with a perfect matching deleted.  No set of size $r$ percolates in such a graph and the minimum degree is $n-2 = \left\lfloor \frac{r-1}{r}n \right\rfloor$.  In the special case of $r = 2$, it is noted in \cite{FPR15} that for any $n$, a graph consisting of two disjoint cliques on $\lfloor n/2 \rfloor$ vertices and $\lceil n/2 \rceil$ vertices has minimum degree $\left\lfloor \frac{(2-1)}{2}n\right\rfloor - 1$ and no set of size $2$ that percolates in $2$-neighbour bootstrap percolation. Though it is not stated in their paper, the proof idea in~\cite{FPR15} can be used, with a small extra check, to show that for $n$ sufficiently large, and $\delta(G) \geq \lfloor n/2 \rfloor$, then $m(G, 2) = 2$.  

Freund, Poloczek, and Reichman~\cite{FPR15} further investigated Ore-type degree conditions for a graph that guarantee that $m(G, 2) = 2$.  Defining $\sigma_2(G)$ to be the minimum sum of degrees of non-adjacent vertices in $G$, they showed that for a graph on $n \geq 2$ vertices,  if $\sigma_2(G) \geq n$, then $m(G, 2) =2$.  Recently, Dairyko, Ferrara, Lidick\'{y}, Martin, Pfender, and Uzzell~\cite{DFLMPU16} improved this result showing that, except for a list of exceptional graphs that they completely characterized, if $\sigma_2(G) \geq n-2$, then $m(G, 2) > 2$.  Their results show that the only graph with $\delta(G) = \lfloor |V(G)|/2 \rfloor$ and $m(G, 2) > 2$ is the $5$-cycle.

The examples showing the tightness of results on the minimum degree in \cite{FPR15} are only given for a small value of $n$ depending on $r$.  When $r \geq 3$ and the number of vertices is large relative to $r$, a different picture emerges and, in fact, when $n$ is large, any graph on $n$ vertices with a minimum degree that exceeds $n/2$ by some constant that depends on $r$ will have a set of size $r$ that percolates in $r$-neighbour bootstrap percolation.  The main result of this paper is the following.

\begin{theorem}\label{thm:main-ub}
For any $r \geq 4$ and $n$ sufficiently large, if $G$ is a graph on $n$ vertices with $\delta(G) \geq \lfloor n/2 \rfloor +(r-3)$, then $m(G, r) = r$.
\end{theorem}

The result for the case $r = 3$ is slightly different than the rest and is, perhaps, closer to the behaviour of the case $r = 2$ examined in~\cite{FPR15}.

\begin{theorem}\label{thm:ub-3}
For any $n \geq 30$, any graph $G$ on $n$ vertices with $\delta(G) \geq \lfloor n/2 \rfloor + 1$ satisfies $m(G, 3) = 3$.
\end{theorem}

In both Theorem~\ref{thm:main-ub} and Theorem~\ref{thm:ub-3}, no attempt has been made to optimize the possible  lower bounds on $n$.

While it remains true that a graph consisting of two disjoint cliques of size $\lfloor n/2\rfloor$ and $\lceil n/2 \rceil$ will have no set of size $r$ that percolates in $r$-neighbour bootstrap percolation, for large $n$, graphs with larger minimum degree exist with no small percolating sets.  In Section~\ref{sec:construction}, examples are given of graphs on $n$ vertices with $\delta(G) = \lfloor n/2 \rfloor$ and $m(G, 3) > 3$ and for every $r \geq 4$, examples of graphs with $\delta(G) = \lfloor n/2 \rfloor + (r-4)$ and $m(G, r) > r$.   These examples show that Theorem~\ref{thm:main-ub} and Theorem~\ref{thm:ub-3} are sharp. 

Throughout, the following notation is used.  Given two disjoint sets of vertices $A$ and $B$ in a graph $G$, let $e(A, B)$ denote the number of edges with one endpoint in $A$ and the other in $B$.  The subgraph of $G$ induced by the set $A$ is denoted by $G[A]$ and given two disjoint sets $A$ and $B$, let $G[A, B]$ denote the bipartite subgraph consisting of all the edges in $G$ with one endpoint in $A$ and the other in $B$.  Given a set $A$ and a vertex $x$, let $\deg_A(x)$ be the number of neighbours of $x$ in the set $A$.  The neighbourhood of a vertex $x$ in $G$ is denoted $N(x)$.

The remainder of the paper is organized as follows.  In Section~\ref{sec:construction}, we give the classes of graphs that show the sharpness of Theorem~\ref{thm:main-ub} and Theorem~\ref{thm:ub-3}.  In Section~\ref{sec:extremal}, it is shown that for all large graphs satisfying the degree conditions of Theorem~\ref{thm:main-ub} or Theorem~\ref{thm:ub-3}, every closed set is either relatively small, consists of around half the vertices, or is the set of all vertices.  Using the existence of small complete bipartite subgraphs, it is shown that there is always a set of $r$ vertices whose closure is not too small.  In Section~\ref{sec:large-closed}, it is shown that graphs with closed sets consisting of nearly half the vertices are highly structured and that this structure can be exploited to find a percolating set of size $r$.  Finally, in Section~\ref{sec:open}, some further open problems are given.

\section{Graphs with no small percolating sets}\label{sec:construction}

The graphs described in this section showing the sharpness of Theorem~\ref{thm:main-ub} and Theorem~\ref{thm:ub-3} consist of two disjoint cliques, with a regular (or nearly-regular when the number of vertices is odd) bipartite graph between them.

\begin{theorem}\label{thm:const-r}
For $r \geq 4$, let $n \geq 2(r-1)$ be even and suppose that $H$ is any $(r-3)$-regular bipartite graph with no $4$-cycles on parts $A$ and $B$ of size $n/2$ each.  The graph $G$ consisting of $H$ together with a clique on the vertices of $A$ and a clique on the vertices of $B$ has $\delta(G) = n/2 + (r-4)$ and $m(G, r) > r$.
\end{theorem}

\begin{proof}
As the graph $G$ is $(n/2+(r-4))$-regular, it remains only to show that no set of $r$ vertices percolates.

Let $X$ be any initially infected set of $r$ vertices in $G$ and set $|X\cap A| = k$.  Note that since every vertex in $A$ has $r-3$ neighbours in $B$ and every vertex in $B$ has $r-3$ neighbours in $A$, then if either $k \leq 2$ or $r-k \leq 2$, some vertices in one partition set will never have $r$ infected neighbours, even if all vertices in the other partition set are infected.

We first use this observation to deal with some of the small values of $r$.
For $r \in \{4, 5\}$, if $k \geq 3$, then $r - k \leq r-3 \leq 2$.  Thus, in the cases $r = 4$ or $r=5$, it is immediate that $X$ does not percolate and so $m(G, r) > r$.

Next, consider the case that $r =6$.  By the previous observation and relabelling $A$ and $B$ if necessary, assume that $3 \leq k \leq r-k \leq r-3$, so that we have $3 = k = r-k$.  If anything further is infected by $X$, say $a \in A$, then $a$ must be adjacent to all $3$ elements of $X \cap B$.  Since $H$ contains no copies of $C_4$, no other vertices in $A$ can be adjacent to all elements of $X\cap B$ and so there is at most one such $a \in A$.  

If $a$ is the only vertex infected at time $1$, then no vertex in $B$ is adjacent to all elements of $X \cap A$ (or else it would have been infected in the first time step) and the only vertices adjacent to $a$ are those in $X \cap B$, which are already infected.  Thus, nothing further is infected.

If two vertices are infected at the first time step, it can only be that one $a \in A$ and one $b \in B$ are infected.  That is, $a$ is adjacent to all elements in $X \cap B$ and $b$ is adjacent to all elements in $X \cap A$.  At the second time step, any further vertex in $A$ is adjacent to $4$ infected vertices in $A$, but not $b$ and at most one from $X \cap B$ and so does not become infected.  Similarly, nothing in $B$ that is not already infected has more than $5$ infected neighbours.  Thus, $X$ does not percolate and so $m(G, 6) > 6$.

Now we consider the most general case: $r \geq 7$.  As above, let $X$ be any set of $r$ vertices in $G$, set $k = |X \cap A|$ and assume that $3 \leq k, r-k \leq r-3$.

\begin{center}
\begin{figure}[htb]
	$(a)$\begin{tikzpicture}[scale = 0.5]
	\tikzstyle{vertex}=[circle, draw=black,  minimum size=4pt,inner sep=0pt]
	
	\filldraw[draw = black, fill = gray!20] (0, 2.5) ellipse (1cm and 3.5cm);
	\filldraw[draw = black, fill = gray!20] (4, 2.5) ellipse (1cm and 3.5cm);
	\node at (0, -1.5) {$A$};
	\node at (4, -1.5) {$B$};
	
	\draw (0, 4.25) ellipse (0.5cm and 1.25cm); 
	\node at (-2.3, 4.25) {$X \cap A$};
	
	\draw (4, 4) ellipse (0.5cm and 1.5cm);
	\node at (6.3, 4) {$X \cap B$};
	
	\node[vertex, fill = black] at (0, 5.25) (a1) {};
	\node[vertex, fill = black] at (0, 4.75) (a2) {};
	\node at (0, 4.3) {\small $\vdots$};
	\node[vertex, fill = black] at (0, 3.75) (a3) {};
	\node[vertex, fill = black] at (0, 3.25) (a4) {};
	
	\node[vertex, fill = white, label=below:{$x$}] at (0, 0) (x) {};
	
	\node[vertex, fill = black] at (4, 5.25) (b1) {};
	\node[vertex, fill = black] at (4, 4.75) (b2) {};
	\node[vertex, fill = black] at (4, 4.25) (b3) {};
	\node at (4, 3.9) {\small $\vdots$};
	\node[vertex, fill = black] at (4, 3.25) (b4) {};
	\node[vertex, fill = black] at (4, 2.75) (b5) {};
	
	\draw (4, 1) ellipse (0.5cm and 1cm);
	\node at (5.5, 1) {$N_x$};
	\node[vertex, fill = white, label=below:{$y$}] at (4, 1) (y) {};
	\draw (x) -- (y);
	\draw (x) -- (3.9, 1.98);
	\draw (x) -- (4, 0);
	
	\draw (x) -- (b1);
	\draw (x) -- (b2);
	\draw (x) -- (b3);
	\draw (x) -- (b4);
	\draw (x) -- (b5);
	\draw (x) -- (3.7, 3.5);
	
	\draw[dashed] (y) -- (a1);
	\draw (y) -- (a2);
	\draw (y) -- (a3);
	\draw (y) -- (a4);
	\draw (y) -- (0.25, 4);
	
	\draw (0, 1.75) ellipse (0.5cm and 1 cm);
	\draw (y) -- (0.1, 2.73);
	\draw (y) -- (0, 0.75);
	
	%\node at (2, -4) {$(a)$};
	
\end{tikzpicture} 
%\hspace*{5pt}
$(b)$ \begin{tikzpicture}[scale=0.5]
	\tikzstyle{vertex}=[circle, draw=black,  minimum size=4pt,inner sep=0pt]
	
	\filldraw[draw = black, fill = gray!20] (0, 2.5) ellipse (1cm and 3.5cm);
	\filldraw[draw = black, fill = gray!20] (4, 2.5) ellipse (1cm and 3.5cm);
	\node at (0, -1.5) {$A$};
	\node at (4, -1.5) {$B$};
	
	\draw (0, 4.25) ellipse (0.5cm and 1.25cm); 
	\node at (-2.3, 4.25) {$X \cap A$};
	
	\draw (4, 4) ellipse (0.5cm and 1.5cm);
	\node at (6.3, 4) {$X \cap B$};

%	\draw[rounded corners] (-0.75, 3) rectangle (0.75, 5.5);
%	\node at (-2.5, 4.25) {$|X \cap A| = k$};
%	
%	\draw[rounded corners] (3.25, 2.5) rectangle (4.75, 5.5);
%	\node at (6.75, 4) {$|X \cap B| = r-k$};
	
	\node[vertex, fill = black] at (0, 5.25) (a1) {};
	\node[vertex, fill = black] at (0, 4.75) (a2) {};
	\node at (0, 4.3) {\small $\vdots$};
	\node[vertex, fill = black] at (0, 3.75) (a3) {};
	\node[vertex, fill = black] at (0, 3.25) (a4) {};
	
	\node[vertex, fill=white, label=below:{$x$}] at (0, 0) (x) {};
	
	\node[vertex, fill = black] at (4, 5.25) (b1) {};
	\node[vertex, fill = black] at (4, 4.75) (b2) {};
	\node[vertex, fill = black] at (4, 4.25) (b3) {};
	\node at (4, 3.9) {\small $\vdots$};
	\node[vertex, fill = black] at (4, 3.25) (b4) {};
	\node[vertex, fill = black] at (4, 2.75) (b5) {};
	
	\draw (4, 1.5) ellipse (0.5cm and 0.75cm);
	\node at (4, 1.5) {$N_x$};
	\node[vertex, fill=white, label=below:{$y$}] at (4, 0) (y) {};
	\draw (x) -- (3.8, 2.19);
	\draw (x) -- (4, 0.75);
	
	\draw (x) -- (b1);
	\draw (x) -- (b2);
	\draw (x) -- (b3);
	\draw (x) -- (b4);
	\draw (x) -- (b5);
	\draw (x) -- (3.7, 3.5);
	
	\draw (y) -- (a1);
	\draw (y) -- (a2);
	\draw (y) -- (a3);
	\draw (y) -- (a4);
	\draw (y) -- (0.25, 4);
	
	\draw (0, 1.75) ellipse (0.5cm and 1 cm);
	\node at (0, 1.75) {$N_y$};
	\draw (y) -- (0.15, 2.71);
	\draw (y) -- (0, 0.75);
	
\end{tikzpicture}

	\caption{Possible structures for the set $X$ when $r \geq 7$ if $(a)$ one vertex is infected in the first time step or $(b)$ two vertices are infected in the first time step.  Shaded regions represent cliques.}\label{fig:Ext-example-r7}
\end{figure}
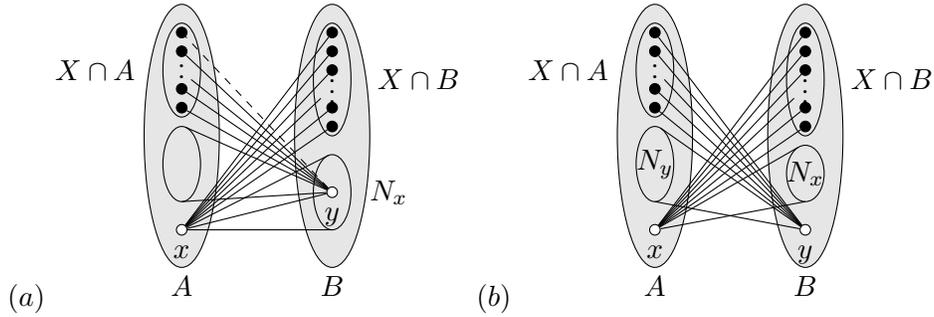
\end{center}

First suppose that only vertices in one partition set, say $A$, are infected at the first time step.  Since $H$ has no copy of $C_4$, there can be only one such vertex $x$ adjacent to all $r-k$ vertices in $X \cap B$.  At the second time step, no vertex in $A$ can be infected as it would have to have $r-k-1 \geq 2$ neighbours in $X \cap B$, which would create a $C_4$ with $x$.  Any vertex in $B$ that is infected at time $2$ is in $N(x) \cap B \setminus X$.  Set $N_x = N(x) \cap B \setminus X$ and note that $|N_x| = k-3$.  If $k = 3$, then no further vertices are infected and the process stops.  If $k \geq 4$ and $y \in N_x$ is infected at the second time step, then $y$ has exactly $k-1$ neighbours in $X \cap A$ and there can only be one such vertex since a second would have two common neighbours with $y$ in $A$.  See Figure~\ref{fig:Ext-example-r7} $(a)$.  At time step $3$, any vertex in $A$ has at most $k+1$ infected neighbours in $A$ and at most $1$ infected neighbour in $B$ (since two would create a $C_4$ with $x$ in $H$).  Any vertex in $B$ has at most $(r-k) + 1$ infected neighbours in $B$ and is either adjacent to $x$ and at most one vertex from $X \cap A \setminus N(y)$ or else at most two vertices from $X\cap A$.  Thus, any uninfected vertex is adjacent to at most $\max\{k+2, r-k+3\}$ infected vertices.  Since $k \geq 4$, then $\max\{k+2, r-k+3\} \leq r-1$ and so no further vertices are infected.

Next, suppose that $x \in A$ and $y \in B$ are both infected at the first time step.  Without loss of generality, assume that $3 \leq k \leq r-k \leq r-3$.  As before, there can be only one vertex in each partition set that is infected in the first time step.  Set $N_x = B \cap N(x) \setminus X$ and $N_y = A \cap N(y) \setminus X$ so that $|N_x| = k-3$ and $|N_y| = r-k-3$.  See Figure~\ref{fig:Ext-example-r7} $(b)$.  At time step $2$, any vertex in $N_y$ is adjacent to $k+1$ infected vertices in $A$ at most $2$ vertices in $B$ ($y$ and at most one from $X\cap B$).  Since $k+1+2 = k+3 \leq r-1$, then such a vertex is not infected.  If $k = 3$, then there are no vertices in $N_x$ and so any vertex in $B$ has at most $r-k+2 \leq r-1$ infected neighbours and so is not infected.  If $k \geq 4$, then any vertex in $N_x$ is adjacent to at most $(r-k) + 1$ infected neighbours in $B$ and at most $2$ in $A$ since it can have at most one neighbour in $X \cap A$.  Since $r-k+3 \leq r-1$, then such a vertex is not infected.

%Finally, consider the case that only vertices from $B$ are infected in the first time step.  As before, there can be only one such vertex in $B$ and so call it $y$ and set $N_y = N(y) \cap A \setminus X$.  Note that $|N_y| = r-k-3$ and any further vertex that becomes infected at the next step must be in the set $N_y$.  If $x \in N_y$ is infected at time $2$, then $x$ has at least $r-k-1$ neighbours in $X \cap B$ and since $r-k \geq 4$, there is at most one such vertex.  Then, at time $3$, any uninfected vertex in $B$ can have at most $1$ neighbour in $X\cap A \cup \{x\}$ and has $r-k+1$ infected neighbours in $B\cap X$.  Since $r-k+1 \leq r-1$, such a vertex is not infected.  For any vertex $w \in A$ that is uninfected, if $w \in N_y$, then $w$ has $k+1$ infected neighbours in $A$, is adjacent to $y$, and among the elements of $X \cap B$, can only be adjacent to the single non-neighbour of $x$.  Thus, $w$ has at most $k+3 \leq r-1$ infected neighbours.  Lastly, if $w \notin N_y$, then $w$ has at most $2$ neighbours in $X \cap B$ (the non-neighbour of $x$ and one other element) and so has at most $k+1 + 2 \leq r-1$ infected neighbours.

As the set $X$ was arbitrary and in all cases, the bootstrap process halts with not all vertices infected, $m(G, r) > r$.

\end{proof}

Note that a graph $H$ with girth at least $6$ as required by Theorem~\ref{thm:const-r} is given with positive probability by taking a random $(r-3)$-regular graph on two vertex sets of size $n/2$ as long as $n$ is sufficiently large (see, for example, \cite{nW78} and \cite{nW99}).

\begin{corollary}\label{cor:const-r}
For every $r \geq 4$ and $n$ sufficiently large, there is a graph $G$ on $n$ vertices with $\delta(G) = \lfloor n/2 \rfloor + (r-4)$ with $m(G, r) > r$.
\end{corollary}

\begin{proof}
If $n$ is even, let $G$ be given by Theorem~\ref{thm:const-r}.  If $n$ is odd, let $G_1$ be the graph on $n+1$ vertices  given by Theorem~\ref{thm:const-r} and define a graph $G$ by deleting one vertex from $G_1$.  The vertices of $G$ are partitioned into a set, $A$, of size $\lceil n/2 \rceil$ that form a clique and a set, $B$, of size $\lfloor n/2 \rfloor$ that each form a clique.  Vertices in $A$ have degree at least $\lceil n/2 \rceil - 1 + (r-3) - 1 = \lfloor n/2 \rfloor + (r-4)$ while vertices in $B$ have degree exactly $\lfloor n/2 \rfloor -1 + (r-3) = \lfloor n/2 \rfloor + (r-4)$.  If $G$ had a percolating set of size $r$ for $r$-neighbour bootstrap percolation, then this same set would percolate in $G_1$ since the additional vertex is joined to at least $r$ neighbours in $B$.  As this would contradict the fact that $m(G_1, r) > r$, then $m(G, r) > r$ also.
\end{proof}

The case $r = 3$ has a different behaviour than larger values of $r$.  The proof that the example has no small percolating sets is closely related to the corresponding proofs for $r \in \{4, 5\}$.

\begin{theorem}\label{thm:const-3}
For any even $n \geq 4$, let $A=[1, n/2]$ and $B = [n/2+1, n]$ and let $G$ be the graph given by a complete graph on $A$, a complete graph on $B$ and a perfect matching between $A$ and $B$.  Then, $\delta(G) = n/2$ and $m(G, 3) > 3$.
\end{theorem}

\begin{proof}
Let $X$ be any set of $3$ vertices in $G$.  Note that either $|X \cap A| \leq 1$ or else $|X \cap B| \leq 1$.  Suppose, without loss of generality that $|X \cap A| \leq 1$.  Even if every vertex in $B$ becomes infected, any uninfected vertex in $A$ has at most $2$ infected neighbours: any vertex in $X \cap A$ and the single neighbour in $B$.  Thus, these vertices never become infected and so $X$ does not percolate. 
\end{proof}

Using the same argument as that given in the proof of Corollary~\ref{cor:const-r} extends Theorem~\ref{thm:const-3} to all $n \geq 4$.

\begin{corollary}
For any $n \geq 4$, there exists a graph $G$ with $\delta(G) = \lfloor n/2 \rfloor$ and $m(G, 3) > 3$.
\end{corollary}

Note that the graph described in Theorem~\ref{thm:const-3} was also used \cite{FPR15} where it was called $DC_n$ and it was noted, in relation to $2$-neighbour bootstrap percolation, that this graph has sets of size $2$ whose closure is of size $n/2$, while there are other sets of size two that percolate.

This concludes the descriptions of constructions and in the subsequent sections, it is shown that large graphs with minimum degree one larger (for a fixed $n$ and $r$) than those in Theorems~\ref{thm:const-r} and \ref{thm:const-3} do have small percolating sets.  No attempt has been made here to classify the extremal examples.

\section{Sets with large closure}\label{sec:extremal}

Before proceeding to the proofs of the main theorems, we give a number of results about the size of the closures of sets in $r$-neighbour bootstrap percolation.  In particular, the goal is to show that the closures of \emph{any} set in graphs satisfying the minimum degree conditions under consideration can only can only have a small number of different sizes.

The following straightforward lemma uses the minimum degree condition to show that any large set will percolate.  This will be used throughout in arguments to come.

\begin{lemma}\label{lem:big-sets-perc}
For any $r \geq 3$, $k \geq 1$, let $G$ be a graph on $n$ vertices with $\delta(G) \geq \lfloor n/2 \rfloor + k$.  Every set $A \subseteq V(G)$ with $|A| \geq \left\lceil \frac{n}{2} \right\rceil +(r-k-1)$ satisfies $\langle A \rangle_r = V(G)$.
\end{lemma}

\begin{proof}
For every $x \in A^c$, since $x$ has at most $|A^c|-1 = n-|A|-1$ neighbours within $A^c$, then
\begin{align*}
\deg_A(x) &= \deg(x) - \deg_{A^c}(x)\\
		&\geq \left\lfloor \frac{n}{2} \right\rfloor + k - n + |A|+1\\
		&\geq \left\lfloor \frac{n}{2} \right\rfloor + k - n + \left(\left\lceil \frac{n}{2} \right\rceil +r-k-1 \right) + 1\\
		&=r.
\end{align*}
Thus, as every vertex in $A^c$ has at least $r$ neighbours in $A$, if the set $A$ is initially infected, the remainder of the graph becomes infected in one time step.
\end{proof}

There are two different cases for the choice of $k$ in Lemma~\ref{lem:big-sets-perc} used here.  In the case $r = 3$ with $k = 1$, this lemma states that if $\delta(G) \geq \lfloor n/2 \rfloor +1$, then any set of size $\lceil n/2 \rceil +1$ percolates.  For all $r \geq 4$, taking $k = r-3$, the lemma shows that for $\delta(G) = \lfloor n/2 \rfloor +r-3$, any set of size $\lceil n/2 \rceil +2$ percolates.

In the following proposition, we consider large graphs with a given minimum degree condition.   Edge-counting is used to show that any set that is closed is either relatively small or else contains nearly half of the vertices of the graph.  This includes the possibility that the set percolates.

\begin{proposition}\label{prop:no-mid-size-closed}
Let $r \geq 3$, set $k  = \max\{1, r-3\}$ and let $G$ be a graph on $n$ vertices with $n \geq 10r$ and $\delta(G) \geq \lfloor n/2 \rfloor + k$.  If $A \subseteq V(G)$ is such that $\langle A \rangle_r = A$, then either $|A| \leq 2(r-1)$ or else $|A| \geq \lfloor n/2 \rfloor - \min\{1, r-3\}$.
\end{proposition}

\begin{proof}
Let $A$ be a set of vertices with $\langle A \rangle_r = A$ and set $|A| = \ell$.  The proof proceeds by counting the edges with one endpoint in $A$ and the other in $A^c$ in two different ways.

Since any vertex in $A$ has at most $\ell-1$ neighbours within the set $A$, any $x \in A$ has at least $\delta(G) - \ell+1$ neighbours in the set $A^c$.  Thus,
\begin{equation}\label{eq:lb-cross-edges}
	e(A, A^c) = \sum_{x \in A} \deg_{A^c}(x) \geq \ell(\delta(G) - \ell+1) \geq \ell(\lfloor n/2 \rfloor - \ell + k+1).
\end{equation}
On the other hand, since $\langle A \rangle_r = A$, every vertex in $A^c$ can have at most $r-1$ neighbours in the set $A$.  Thus,
\begin{equation}\label{eq:ub-cross-edges}
	e(A, A^c)  = \sum_{x \in A^c} \deg_A(x) \leq (r-1)|A^c| = (r-1)(n-\ell).
\end{equation}

Combining the inequalities \eqref{eq:lb-cross-edges} and \eqref{eq:ub-cross-edges} and rearranging gives that
\begin{equation}\label{eq:edge-balance}
0 \leq \ell^2 - \ell\left(\left\lfloor \frac{n}{2} \right\rfloor + k+r \right) + (r-1)n.
\end{equation}

Define $D(\ell) = \ell^2 - \ell\left(\lfloor n/2 \rfloor + k+r \right) + (r-1)n$, which is the righthand side of inequality \eqref{eq:edge-balance}.  Substituting $\ell=2r-1$ into $D(\ell)$ gives
\begin{align*}
D(2r-1)	&=(2r-1)^2 - (2r-1)\lfloor n/2 \rfloor -(2r-1)(k+r) + n(r-1)\\
		&= (2r-1)(r-k-1) + (r-1) \left(n - 2 \left\lfloor \frac{n}{2} \right\rfloor \right) - \left\lfloor \frac{n}{2} \right\rfloor\\
		&\leq (2r-1)(r-k-1) + (r-1) - \left\lfloor \frac{n}{2} \right\rfloor\\
		&=
		\begin{cases}	
			7 - \left\lfloor \frac{n}{2} \right\rfloor 	&\text{if $r = 3$}\\
			5r-3 - \left\lfloor \frac{n}{2} \right\rfloor &\text{if $r \geq 4$}
		\end{cases}\\	
		&< 0
\end{align*}
since for $n \geq 10r-4$, then $\lfloor n/2 \rfloor > 5r-3$.  Furthermore, substituting $\ell = 2r-2$ gives, for all $n$,
\begin{align*}
D(2r-2)	&=(2r-2)^2 - 2(r-1)\lfloor n/2 \rfloor - (2r-2)(k+r) + n(r-1)\\
		&= (2r-2)(r-k-2) + (r-1)\left(n - 2 \left\lfloor \frac{n}{2}\right\rfloor \right) \geq 0.
\end{align*}

Similarly, substituting $\ell = \lfloor n/2 \rfloor -2$ gives
\begin{align*}
D(\lfloor n/2 \rfloor -2)
	&=\lfloor n/2 \rfloor^2 - 4\lfloor n/2 \rfloor + 4 - \lfloor n/2 \rfloor^2 - (k+r-2)\lfloor n/2 \rfloor\\
	& \qquad + 2(k+r) + n(r-1)\\
	&=(r-1)\left( n - 2 \lfloor n/2 \rfloor \right) - (k-r+4)\lfloor n/2 \rfloor + 2(k+r+2)\\
	&\leq 2k+3r+3 - \lfloor n/2 \rfloor < 0
\end{align*}
for $n \geq 10r$.  Next consider the result of substituting $\ell = \lfloor n/2 \rfloor -1$,
\begin{align*}
D(\lfloor n/2 \rfloor -1)
	&=\lfloor n/2 \rfloor^2 - 2\lfloor n/2 \rfloor + 1 - \lfloor n/2 \rfloor^2 - (k+r-1)\lfloor n/2 \rfloor\\
	& \qquad + (k+r) + n(r-1)\\
	&=n(r-1) - (k+r+1)\lfloor n/2 \rfloor + (k+r+1)\\
	&=\begin{cases}
		2n - 5\lfloor n/2 \rfloor + 5		&\text{if $r = 3$}\\
		(r-1)\left(n - 2\lfloor n/2 \rfloor\right) + 2r-2	&\text{if $r \geq 4$}.
	\end{cases}
\end{align*}
Thus, when $r = 3$,  and $n \geq 16$, $D(\lfloor n/2 \rfloor -1)< 0$ whereas for $r \geq 4$ and all $n$, we have $D(\lfloor n/2 \rfloor -1) \geq 0$.  Finally, consider $D(\lfloor n/2 \rfloor)$ in the case that $r = 3$:
\begin{align*}
D(\lfloor n/2 \rfloor)
	&=\lfloor n/2 \rfloor^2 - \lfloor n/2 \rfloor^2 - 4\lfloor n/2 \rfloor + 2n\\
	&=2\left(n - 2\lfloor n/2 \rfloor \right) \geq 0.
\end{align*}

Note that $D$ is a quadratic function in $\ell$ with a unique minimum and satisfying $D(2r-2) \geq 0$, $D(2r-1) < 0$.  When $r = 3$, since $D(\lfloor n/2 \rfloor - 1) < 0$ and $D(\lfloor n/2 \rfloor) \geq 0$, then if $D(\ell) \geq 0$, then either $\ell \leq 4$ or else $\ell \geq \lfloor n/2 \rfloor$.  When $r \geq 4$, since $D(\lfloor n/2 \rfloor - 2) < 0$ and $D(\lfloor n/2 \rfloor - 1) \geq 0$, then if $D(\ell) \geq 0$, either $\ell \leq 2(r-1)$ or else $\ell \geq \lfloor n/2 \rfloor -1$.  This completes the proof.
\end{proof}

In summary, Lemma~\ref{lem:big-sets-perc} and Proposition~\ref{prop:no-mid-size-closed} together show that if $G$ is a graph on $n$ vertices with $\delta(G) \geq \lfloor n/2 \rfloor + \max\{1, r-3\}$, then any set of $r$ vertices either percolates, spans a set of size at most $2(r-1)$ or else spans a set of cardinality close to $n/2$.

In order to address the existence of small closed sets of vertices in the graph, note that for $r$ fixed and $n$ large enough, the K\"{o}vari-S\'{o}s-Tur\'{a}n theorem~\cite{KST54} implies that a graph on $n$ vertices with minimum degree $\delta(G) \geq \lfloor n/2 \rfloor + (r-3)$ contains complete bipartite subgraphs of the form $K_{r, r-1}$ which give a subgraph on $2r-1$ vertices with $m(K_{r, r-1}, r) = r$.  For the sake of completeness, the following pair of lemmas with standard proofs make this precise.

\begin{lemma}\label{lem:complete-bip-3}
For $n \geq 6$, if $G$ is a graph on $n$ vertices with $\delta(G) \geq \lfloor n/2 \rfloor +1$, then any vertex of $G$ is contained in a copy of $K_{2, 3}$.
\end{lemma}

\begin{proof}
Let $x$ be any vertex in $G$.  If $x$ is adjacent to all other vertices, then for any $y \neq x$, the common neighbourhood of $x$ and $y$ has at least $\lfloor n/2 \rfloor \geq 3$ vertices and these together with $x$ and $y$ form a copy of $K_{2, 3}$.  Otherwise, let $z$ be any non-neighbour of $x$.  Then the common neighbourhood of $x$ and $z$ has at least $2(\lfloor n/2 \rfloor + 1) - (n-2) = 2\lfloor n/2 \rfloor - n + 4 \geq 3$ vertices and these together with $x$ and $z$ form a copy of $K_{2, 3}$.
\end{proof}

\begin{lemma}\label{lem:complete-bip-r}
For each $r \geq 3$ and $n \geq (r-1)2^{r-1} +4$, if $G$ is a graph on $n$ vertices with $\delta(G) \geq \lfloor n/2 \rfloor + (r-3)$, then $G$ contains a copy of $K_{r, r-1}$.
\end{lemma}

\begin{proof}
The proof proceeds by counting copies of stars of the form $K_{1, r-1}$.  Define the set
\[
S = \{(x, A) \mid x \in V(G),\ |A| = r-1,\ A \subseteq N(x)\}.
\]
Then, counting elements of $S$ by the first coordinate, as long as $\lfloor n/2 \rfloor + (r-2) \geq r-1$, then
\begin{align*}
|S|	&=\sum_{x \in V} \binom{\deg(x)}{r-1} \\
	& \geq \sum_{x \in V} \binom{\lfloor n/2 \rfloor + (r-3)}{r-1}\\
	& = n \binom{\lfloor n/2 \rfloor + (r-3)}{r-1}\\
	& \geq \frac{n}{(r-1)!}\left(\frac{n-1}{2} + (r-3)\right)\cdots \left(\frac{n-1}{2}+1\right)\left(\frac{n-1}{2}\right)\left(\frac{n-1}{2} - 1\right)\\
	&\geq \frac{n}{(r-1)!} \cdot \frac{(n-1)(n-2) \cdots (n-r+2)}{2^{r-1}} \cdot (n-3)\\
	&\geq \frac{n-3}{2^{r-1}} \binom{n}{r-1} > (r-1) \binom{n}{r-1}.
\end{align*}
As there are $\binom{n}{r-1}$ possible choices for the second coordinate of elements of $S$, by the pigeonhole principle, there is a set $A \subseteq V(G)$ of $r-1$ vertices with at least $r$ common neighbours.  These $r$ vertices, together with $A$ contain a copy of $K_{r, r-1}$ in the graph.
\end{proof}

Thus, when $n$ is sufficiently large, any graph on $n$ vertices with minimum degree $\lfloor n/2 \rfloor + (r-3)$ has a set of size $r$ whose span contains at least $2r-1$ vertices and hence by Proposition~\ref{prop:no-mid-size-closed}, contains at least $\lfloor n/2 \rfloor -1$ vertices.

What remains to show is that if such a graph contains a set $A$ with around half the vertices in $G$ and $\langle A \rangle_r = A$, then $G$ contains some set of size $r$ that percolates.

\section{Structure of large closed sets}\label{sec:large-closed}

In this section, we show that if a graph on $n$ vertices has minimum degree $\lfloor n/2 \rfloor + \max\{r-3, 1\}$ and a set $A$ with $\langle A \rangle_r = A$ and $A$ has close to half the vertices of $G$, then enough structural information about $G$ can be deduced to show that there is some set of size $r$ that percolates, completing the proof of Theorem~\ref{thm:main-ub}.  As the minimum degree conditions for the case $r = 3$ are different from all others, these are dealt with separately.

Before proceeding with these results, a straightforward lemma is recorded to be used repeatedly.  If the minimum degree of a graph is large enough, not only is there a set of $r$ vertices that percolates in $r$-neighbour bootstrap percolation, but, in fact, any set of $r$ vertices will percolate.

\begin{lemma}\label{lem:all-sets-perc}
Let $k \geq 0$, $r \geq 3$ and $n \geq k(r+1)-1$.  For any graph $G$ on $n$ vertices with $\delta(G) \geq n-k$ and any set $A \subseteq V(G)$ of $r$ vertices, $\langle A \rangle_r = V(G)$.
\end{lemma}

\begin{proof}
As each vertex in $G$ has at most $k-1$ non-neighbours, the vertices in the set $A$ have at least $n-|A| - (k-1)|A| =n-kr$ common neighbours.  Since $n-kr \geq k-1$, when the set $A$ is initially infected at least $k-1$ further vertices are infected at the first time step.  At this point, any uninfected vertex is adjacent to at least $(r+k-1) - (k-1) = r$ infected vertices and so becomes infected in the second time step.  Thus, the set $A$ percolates.
\end{proof}

\subsection{Threshold $r = 3$}\label{sec:structure-3}

In this subsection, it is shown that if any set of size $3$ in a graph with minimum degree $\lfloor n/2 \rfloor +1$ spans either $\lfloor n/2 \rfloor$ or $\lceil n/2 \rceil$ vertices, then some set of $3$ vertices percolates in $3$-neighbour bootstrap percolation.  The two different cases that arise when $n$ is odd are handled in separate propositions.

\begin{proposition}\label{prop:small-half-3}
Let $G$ be a graph on $n \geq 13$ vertices with $\delta(G) \geq \lfloor n/2 \rfloor +1$ and let $A \subseteq V(G)$ be such that $|A| = \lfloor n/2 \rfloor$.  If $\langle A \rangle_3 = A$, then $m(G, 3) = 3$.
\end{proposition}

\begin{proof}
Since any vertex $x \in A$ has at most $\lfloor n/2 \rfloor -1$ neighbours within $A$, then 
\[
\deg_{A^c}(x) \geq \lfloor n/2 \rfloor + 1 - \left(\lfloor n/2 \rfloor -1 \right) = 2.
\]
Since $\langle A \rangle_3 = A$, then any vertex $y \in A^c$ has at most $2$ neighbours in $A$ and so
\[
\deg_{A^c}(y) \geq \lfloor n/2 \rfloor + 1 - 2 = \lfloor n/2 \rfloor -1 \geq \left(\lceil n/2 \rceil -1\right) - 1.
\]
Then, by Lemma~\ref{lem:all-sets-perc}, any set of $3$ vertices in $A^c$ infects all of $A^c$.

Set $A_3 = \{x \in A \mid \deg_{A^c}(x) \geq 3\}$.  If $A_3 \neq \emptyset$, then for any three vertices $a, b, c \in A^c$,
\[
|\langle \{a, b, c\} \rangle_3| \geq |A^c| + |A_3| \geq \lceil n/2 \rceil +1.
\]
Then, by Lemma~\ref{lem:big-sets-perc}, $\langle \{a, b, c\} \rangle_3 = V(G)$.

Thus, assume that $A_3 = \emptyset$ and hence $G[A]$ is a complete graph with every vertex having exactly $2$ neighbours in $A^c$.  Since any vertex in $A^c$ has at most $\lceil n/2 \rceil -1$ neighbours within $A^c$, then every vertex in $A^c$ has at least $1$ neighbour in $A$.  Set $B_1 = \{y \in A^c \mid \deg_A(y) = 1\}$.  Since $A$ is closed, every vertex in $A^c$ has at most $2$ neighbours in $A$.  Then, 
\[
2\lfloor n/2 \rfloor = e(A, A^c) = |B_1| + 2(\lceil n/2 \rceil - |B_1|) = 2\lceil n/2 \rceil - |B_1|
\]
which implies that $|B_1| = 2(\lceil n/2 \rceil - \lfloor n/2 \rfloor) \leq 2$.

\begin{center}
\begin{figure}[htb]
\begin{tikzpicture}[scale = 0.75]
	\tikzstyle{vertex}=[circle, draw=black,  minimum size=5pt,inner sep=0pt]
	
	\filldraw[draw = black, fill = gray!20] (0, 2) ellipse (1cm and 3cm);
	\filldraw[draw = black, fill = white] (4, 2) ellipse (1cm and 3cm);
	\node at (0, -1.5) {$A$};
	\node at (4, -1.5) {$A^c$};
	
	\node[vertex, fill = white, label=above:{$a$}] at (0, 4) (a) {};
	\node[vertex, fill = white, label=below right:{$b$}] at (0, 3) (b) {};
	\node[vertex, fill = black, label=below:{$c$}] at (0, 0) (c) {};
	\node[vertex, fill = black, label=right:{$y$}] at (4, 4) (y) {};
	\node[vertex, fill =black, label=right:{$z$}] at (4, 3) (z) {};
	
	\draw (z) -- (a) -- (y) -- (b) -- (c);
	\draw (a) -- (b);
	\draw (c) to[bend left = 30] (a);
	
\end{tikzpicture}
\caption{Case for $r = 3$ and $A_3 = \emptyset$ in the proof of Proposition~\ref{prop:small-half-3}.}\label{fig:r3-floor}
\end{figure}
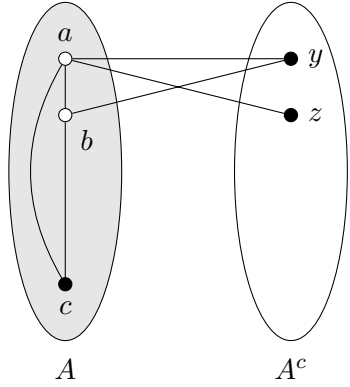
\end{center}

Pick any $y \in A^c \setminus B_1$ and let $a, b$ be its $2$ neighbours in $A$.  Let $z \in A^c$ be any other neighbour of $a$ and choose any $c \in A \setminus \{a, b\}$.  Consider the effect of initially infecting the set $\{c, y, z\}$; see Figure~\ref{fig:r3-floor}.  Then, $a$ is adjacent to all $3$ and becomes infected in the first time step.  Then, $b$ is adjacent to $a$, $c$, and $y$ and so becomes infected by the second time step.  Since $G[A]$ is complete and contains three infected vertices, all remaining vertices of $A$ are infected by the third time step.  Finally, any vertex in $A^c \setminus B_1$ is adjacent to at least one of $y$ and $z$ and has two further infected neighbours in $A$ and so becomes infected by time step $4$.  Finally, if there is a vertex in $B_1$, it is adjacent to all elements of $A^c$ and has one infected neighbour in $A$ and so also becomes infected by step $4$.
\end{proof}

\begin{proposition}\label{prop:big-half-3}
Let $n \geq 13$ be odd and let $G$ be a graph on $n$ vertices with $\delta(G) \geq \frac{n+1}{2}$ and let $A \subseteq V(G)$ be such that $|A| =\frac{n+1}{2}$.  If $\langle A \rangle_3 = A$, then $m(G, 3) = 3$.
\end{proposition}

\begin{proof}
Counting degrees as in the previous proof, for every $y \in A^c$, since $\langle A \rangle_3 = A$, then $y$ has at most $2$ neighbours in $A$ and so at least $\frac{n+1}{2} - 2 = |A^c| -1$ neighbours in $A^c$.  That is, $G[A^c]$ is a complete graph and every vertex has exactly $2$ neighbours in $A$.

Any vertex in $A$ has at least $1$ neighbour in $A^c$.  Set $A_3 = \{x \in A \mid \deg_A(x) \geq 3\}$.  If $|A_3| \geq 2$, then any three vertices in $A^c$ span at least $|A^c| + |A_3| \geq \frac{n-1}{2} + 2 = \lceil n/2 \rceil +1$ vertices and so percolate by Lemma~\ref{lem:big-sets-perc}.  If $A_3 = \emptyset$, then $\langle A^c \rangle_3 = A^c$ and so by Proposition~\ref{prop:small-half-3}, $G$ has a percolating set of size $3$.

Assume now that $|A_3| = 1$.  Note that every vertex in $A\setminus A_3$ has either $1$ or $2$ neighbours in $A^c$ and at most one non-neighbour in $A$.  Thus, by Lemma~\ref{lem:all-sets-perc}, any set of size $3$ in $A \setminus A_3$ eventually infects all of $A \setminus A_3$.  If $\langle A \setminus A_3 \rangle_3 = A \setminus A_3$, then again by Proposition~\ref{prop:small-half-3}, $G$ has a percolating set of size $3$.  

Therefore, assume further that $|A_3| = 1$ and that $\langle A\setminus A_3 \rangle_3 = A$.  Let $x$ be any vertex in $A \setminus A_3$ and let $a$ be one its neighbours in $A^c$.  Let $y, z \in A \setminus A_3$ be any two neighbours of $x$ and consider the effect of initially infecting $\{a, y, z\}$.  Then since $x$ is adjacent to all $3$, it is infected in the first time step.  By assumption, $\langle \{x, y, z\} \rangle_3 = A$ and so $\langle \{a, y, z\}\rangle_3 \supseteq A \cup \{a\}$, which is a set of size $\lceil n/2 \rceil +1$.  Thus, by Lemma~\ref{lem:big-sets-perc}, the vertices $a, y, z$ percolate.

In all cases, the graph $G$ contains $3$ vertices that percolate and so $m(G, 3) = 3$.
\end{proof}

With these two results, the proof of Theorem~\ref{thm:ub-3} now follows.

\begin{proof}[Proof of Theorem~\ref{thm:ub-3}]
Let $n \geq 30$ and let $G$ be a graph on $n$ vertices with $\delta(G) \geq \lfloor n/2 \rfloor +1$.  By Lemma~\ref{lem:complete-bip-3}, $G$ contains a copy of $K_{2, 3}$.  Let $A$ be a set of $3$ vertices in one of the partition classes in any copy of $K_{2,3}$, since $|\langle A \rangle_r| \geq 5 > 2(3-1)$, by Lemma~\ref{lem:big-sets-perc} and Proposition~\ref{prop:no-mid-size-closed}, either $A$ percolates or else $|\langle A \rangle_3| \in \{\lfloor n/2 \rfloor, \lceil n/2 \rceil\}$.  By Propositions~\ref{prop:small-half-3} and \ref{prop:big-half-3}, if $|\langle A \rangle_3| \in \{\lfloor n/2 \rfloor, \lceil n/2 \rceil\}$, then $G$ contains some set of size $3$ that percolates.  Thus, $m(G, 3) = 3$, which completes the proof.
\end{proof}

\subsection{Threshold $r \geq 4$}\label{sec:structure-r}

In this subsection, we consider bootstrap processes with infection threshold $r \geq 4$ and give the proof of Theorem~\ref{thm:main-ub}.  The proof uses more steps than that for the corresponding result for $r=3$ because of the weaker result for Proposition~\ref{prop:no-mid-size-closed} in the case $r \geq 4$.

\begin{proposition}\label{prop:vsmall-half-r}
Let $r \geq 4$, let $n$ be sufficiently large, and let $G$ be a graph on $n$ vertices with $\delta(G) \geq \lfloor n/2 \rfloor + (r-3)$.  If there is a set $A \subseteq V(G)$ with $|A| = \lfloor n/2 \rfloor -1$ and $\langle A \rangle_r = A$, then $m(G, r) = r$.
\end{proposition}

\begin{proof}
Counting edges as in previous proofs and using the fact that $\langle A \rangle_r = A$, for any vertex $y \in A^c$ 
\[
r-4 \leq (r-3) - \lceil n/2 \rceil + \lfloor n/2 \rfloor \leq \deg_{A}(y) \leq r-1
\]
and $y$ has at most $3$ non-neighbours in the set $A^c$.  Thus, by Lemma~\ref{lem:all-sets-perc}, any set of $r$ vertices in $A^c$ infects all of $A^c$.  Since any vertex in $A$ has at least $r-1 \geq 1$ neighbours in $A^c$, $e(A, A^c) \neq 0$.  Let $b \in A^c$ be any vertex with a neighbour $a \in A$.  Let $v_1, v_2, \ldots, v_{r-1}$ be any neighbours of $b$ in $A^c$ and consider initially infecting the set $\{a, v_1, v_2, \ldots, v_{r-1}\}$.  Since $b$ is adjacent to all $r$ infected vertices, it is infected at the first time step.  Then, by the previous comment, $\{b, v_1, \ldots, v_{r-1}\}$ internally spans the entire set $A^c$.  Since, 
\[
|\langle \{a, v_1, v_2, \ldots, v_{r-1}\} \rangle_r| \geq |A^c \cup \{a\}| = \lceil n/2 \rceil + 1 +1 = \lceil n/2 \rceil +2,
\]
then by Lemma~\ref{lem:big-sets-perc}, the set $\{a, v_1, v_2, \ldots, v_{r-1}\}$ percolates and so $m(G, r) = r$.
\end{proof}

\begin{proposition}\label{prop:floor-half-r}
Let $r \geq 4$, let $n$ be sufficiently large, and let $G$ be a graph on $n$ vertices with $\delta(G) \geq \lfloor n/2 \rfloor + (r-3)$.  If there is a set $A \subseteq V(G)$ with $|A| = \lfloor n/2 \rfloor$ and $\langle A \rangle_r = A$, then $m(G, r) = r$.
\end{proposition}

\begin{proof}
Since $\langle A \rangle_r = A$ and $|A^c| = \lceil n/2 \rceil$, then for every $y \in A^c$,
\[
r-3 \leq \lfloor n/2 \rfloor + (r-3) - (\lceil n/2 \rceil -1) \leq \deg_A(y) \leq r-1
\]
and also $y$ has at most $2$ non-neighbours within $A^c$.  Thus, by Lemma~\ref{lem:all-sets-perc}, any $r$ vertices in $A^c$ infect all of $A^c$.

\begin{center}
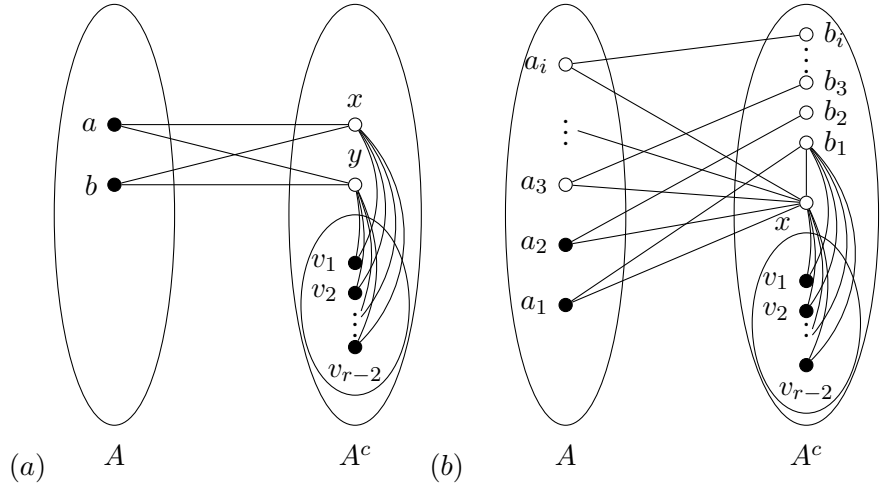
\begin{figure}
$(a)$ \begin{tikzpicture}[scale = 0.8]
	\tikzstyle{vertex}=[circle, draw=black,  minimum size=5pt,inner sep=0pt]
	
%	\draw[rounded corners] (-1, -1) rectangle (1, 5);
%	\draw[rounded corners] (3, -1) rectangle (5, 5);
	\filldraw[draw = black, fill = white] (0, 2.5) ellipse (1cm and 3.5cm);
	\filldraw[draw = black, fill = white] (4, 2.5) ellipse (1.1cm and 3.5cm);
	\node at (0, -1.5) {$A$};
	\node at (4, -1.5) {$A^c$};
	
	\node[vertex, fill = black, label=left:{$a$}] at (0, 4) (a) {};
	\node[vertex, fill = black, label=left:{$b$}] at (0, 3) (b) {};
	\node[vertex, label=above:{$x$}] at (4, 4) (x) {};
	\node[vertex, label=above:{$y$}] at (4, 3) (y) {};
	
	\draw (a) -- (x) -- (b) -- (y) -- (a);
	
	\draw (4, 1) ellipse (0.9cm and 1.5 cm);
	%\node at (6.65, 1) {$N(x) \cap N(y) \cap A^c$};
	
	\node[vertex, fill = black, label=left:{$v_1$}] at (4, 1.7) (v1) {};
	\node[vertex, fill = black, label=left:{$v_2$}] at (4, 1.2) (v2) {};
	\node at (4, 0.8) {$\vdots$};
	\node[vertex, fill = black, label=below:{$v_{r-2}$}] at (4, 0.3) (v3) {};
	
	\draw (y) to[bend left=10] (v1);
	\draw (y) to[bend left = 15] (v2);
	\draw (y) to[bend left=20] (4.1, 0.9);
	\draw (y) to[bend left = 25] (v3);
	
	\draw (x) to[bend left =30] (v1);
	\draw (x) to[bend left = 35] (v2);
	\draw (x) to[bend left = 40] (4.1, 0.8);
	\draw (x) to[bend left = 45] (v3);
	
\end{tikzpicture}
 $(b)$ \hspace*{10pt}
\begin{tikzpicture}[scale = 0.8]
	\tikzstyle{vertex}=[circle, draw=black,  minimum size=5pt,inner sep=0pt]
	
%	\draw[rounded corners] (-1, -1) rectangle (1,6);
%	\draw[rounded corners] (3, -1) rectangle (5, 6);
	\filldraw[draw = black, fill = white] (0, 2.5) ellipse (1cm and 3.5cm);
	\filldraw[draw = black, fill = white] (4, 2.5) ellipse (1.2cm and 3.5cm);
	\node at (0, -1.5) {$A$};
	\node at (4, -1.5) {$A^c$};
	
	\node[vertex, label=right:{$b_1$}] at (4, 3.7) (b1) {};
	\node[vertex, label=below left:{$x$}] at (4, 2.7) (x) {};
	\draw (x) -- (b1);
	
	\node[vertex, label=right:{$b_2$}] at (4, 4.2) (b2) {};
	\node[vertex, label=right:{$b_3$}] at (4, 4.7) (b3) {};
	\node at (4, 5.2) {$\vdots$};
	\node[vertex, label=right:{$b_i$}] at (4, 5.5) (b4) {};
	
	\node[vertex, fill = black, label=left:{$a_1$}] at (0, 1) (a1) {};
	\draw (x) -- (a1) -- (b1);
	\node[vertex, fill = black, label=left:{$a_2$}] at (0, 2) (a2) {};
	\draw (x) -- (a2) -- (b2);
	\node[vertex, label=left:{$a_3$}] at (0, 3) (a3) {};
	\draw (x) -- (a3) -- (b3);
	\node at (0, 4) {$\vdots$};
	\draw (x) -- (0.2, 3.9);
	\node[vertex, label=left:{$a_i$}] at (0, 5) (a4) {};
	\draw (x) -- (a4) -- (b4);
	
	\draw (4, 0.7) ellipse (0.9cm and 1.5 cm);
	%\node at (6.75, 0.7) {$N(x) \cap N(b_1) \cap A^c$};
	
	\node[vertex, fill = black, label=left:{$v_1$}] at (4, 1.4) (v1) {};
	\node[vertex, fill = black, label=left:{$v_2$}] at (4, 0.9) (v2) {};
	\node at (4, 0.8) {$\vdots$};
	\node[vertex, fill = black, label=below:{$v_{r-2}$}] at (4, 0) (v3) {};
	
	\draw (x) to[bend left=10] (v1);
	\draw (x) to[bend left = 15] (v2);
	\draw (x) to[bend left=20] (4.1, 0.6);
	\draw (x) to[bend left = 25] (v3);
	
	\draw (b1) to[bend left =30] (v1);
	\draw (b1) to[bend left = 35] (v2);
	\draw (b1) to[bend left = 40] (4.1, 0.5);
	\draw (b1) to[bend left = 45] (v3);
\end{tikzpicture}	

\caption{$(a)$ $G[A, A^c]$ contains a copy of $K_{2,2}$; $(b)$ $G[A, A^c]$ is $K_{2,2}$-free}\label{fig:r4-small-half}
\end{figure}
\end{center}

If the graph $G[A, A^c]$ contains a copy of $K_{2, 2}$ with vertices $a, b \in A$ and $x, y \in A^c$, let $v_1, v_2, \ldots, v_{r-2}$ be any $r-2$ common neighbours of $x$ and $y$ in $A^c$ and consider initially infecting the set $\{a, b, v_1, v_2, \ldots, v_{r-2}\}$, as in Figure~\ref{fig:r4-small-half} $(a)$.  The vertices $x$ and $y$ are infected in the first time step and subsequently all vertices in $A^c$ are infected.  Since at least $|A^c| + 2 = \lceil n/2 \rceil +2$ vertices are infected, the set percolates by Lemma~\ref{lem:big-sets-perc}.

Now, assume that the graph $G[A, A^c]$ contains no copy of $K_{2,2}$.  Since every vertex $x \in A$ has at least $\lfloor n/2 \rfloor + (r-3) - (\lfloor n/2 \rfloor -1) = r-2$ neighbours in $A^c$, then $e(A, A^c) \geq (r-2)\lfloor n/2 \rfloor$ and so there are at most $\lceil n/2 \rceil/(r-2)$ vertices $y \in A^c$ with $\deg_A(y) = r-3$.  Let $x \in A^c$ be a vertex with $\deg_A(x) = i \in \{r-2, r-1\}$ and let $a_1, a_2, \ldots, a_i$ be its neighbours in $A$.  Note that $i \geq 2$.  As each $a_j$ has at least $r-2 \geq 2$ neighbours in $A^c$, for each $j \leq i$, let $b_j \in A^c \setminus \{x\}$ be a neighbour of $a_j$.  Since $G[A, A^c]$ contains no copy of $K_{2, 2}$ all of the vertices $\{b_1, b_2, \ldots, b_i\}$ are distinct.  Since the vertex $x$ has at most $i - (r-3)$ non-neighbours in $A^c$ and $i - (r-3) \leq i-1$, then $x$ is adjacent to at least one vertex in $\{b_1, b_2, \ldots, b_i\}$.  Without loss of generality, suppose that $x$ is adjacent to $b_1$.  Let $v_1, v_2, \ldots, v_{r-2}$ be any common neighbours of $x$ and $b_1$ in $A^c$ and consider initially infecting the set $\{a_1, a_2, v_1, v_2, \ldots, v_{r-2}\}$, as in Figure~\ref{fig:r4-small-half} $(b)$.  The vertex $x$ is infected in the first time step, and $b_1$ by the second time step.  Then, $A^c$ is internally spanned by $\{x, b_1, v_1, \ldots, v_{r-2}\}$ and since
\[
|\langle \{a_1, a_2, v_1, v_2, \ldots, v_{r-2}\} \rangle_r| \geq |A^c \cup \{a_1, a_2\}| = \lceil n/2 \rceil +2,
\]
then by Lemma~\ref{lem:big-sets-perc}, all vertices are eventually infected and so $m(G, r) = r$.
\end{proof}

The remaining two cases consist of showing that if a set $A$ satisfies $\langle A \rangle_r = A$ and $\lceil n/2 \rceil \leq |A| \leq \lceil n/2 \rceil +1$, then there is a set of size $r$ that percolates.  The following structural fact about graphs with a large closed set is used repeatedly.  The straightforward proof follws the same arguments used in previous propositions in this section.

\begin{fact}\label{fact:big-half-r}
For any $r \geq 4$, $n \geq 2r$ and $G$ a graph on $n$ vertices with $\delta(G) \geq \lfloor n/2 \rfloor +1$, let $A$ be a set with $|A| = \lceil n/2 \rceil +r-3$ and $\langle A \rangle_r = A$.  Then $G[A^c]$ is a complete graph and every vertex has exactly $r-1$ neighbours in $A$.
\end{fact}

The aim in all of the proofs of this section is to use the structural information about the graphs to find a set of $r$ vertices that internally spans at least $\lfloor n/2 \rfloor +2$ vertices (and hence percolates).  In some circumstances, finding many sets whose span is $\lfloor n/2 \rfloor +1$ can be quite useful as Fact~\ref{fact:big-half-r} provides a great deal of information about structure regarding such sets.

\begin{proposition}\label{prop:ceil-half-r}
Let $r \geq 4$, let $n$ be sufficiently large and odd, and let $G$ be a graph on $n$ vertices with $\delta(G) \geq \frac{n-1}{2} + (r-3)$.  If there is a set $A \subseteq V(G)$ with $|A| = \left\lceil \frac{n}{2} \right\rceil = \frac{n+1}{2}$ and $\langle A \rangle_r = A$, then $m(G, r) = r$.
\end{proposition}

\begin{proof}
Since any vertex $y \in A^c$ has at most $\frac{n-1}{2} - 1$ neighbours within $A^c$ and $\langle A \rangle_r = A$, then 
\[
r-2 = \frac{n-1}{2} + (r-3) - \left(\frac{n-1}{2} - 1\right) \leq \deg_A(x) \leq r-1
\]
and so every vertex in $A^c$ has at most $1$ non-neighbour in $A^c$.  As before, by Lemma~\ref{lem:all-sets-perc}, any set of $r$ vertices in $A^c$ infects all of $A^c$, at least.

Since $|A| = \frac{n+1}{2}$, then every vertex in $A$ has at least $r-3 \geq 1$ neighbours in $A^c$.  Set $A_r = \{x \in A \mid \deg_{A^c}(x) \geq r\}$.  Note that since $r|A_r| \leq e(A, A^c) \leq (r-1)|A^c|$, then $|A_r| \leq \frac{(r-1)}{r}\cdot \frac{(n-1)}{2} < \frac{n+1}{2}$ and so $A \setminus A_r \neq \emptyset$.  Any vertex in $A \setminus A_r$ has at most $2$ non-neighbours in $A$ and so any set of size $r$ in $A \setminus A_r$ infects the remainder of $A \setminus A_r$ by Lemma~\ref{lem:all-sets-perc}.

If $A_r = \emptyset$, then $\langle A^c \rangle_r = A^c$, but since $|A^c| = \frac{n-1}{2} = \lfloor n/2 \rfloor$, then by Proposition~\ref{prop:floor-half-r}, there is a set of size $r$ that percolates.

If $|A_r| \geq 2$, then choose any element $a \in A \setminus A_r$, let $b \in A^c$ be any neighbour of $a$ and let $v_1, v_2, \ldots, v_{r-1}$ be any $r-1$ neighbours of $b$ in $A^c$.  Then, letting $B = \{a, v_1, v_2, \ldots, v_{r-1}\}$ be the set of initially infected vertices, $B$ infects $b$ and so with $r$ infected vertices in $A^c$, $|\langle B \rangle_r| \geq |A^c \cup A_r \cup \{a\}| \geq \frac{n-1}{2} + 2+1 = \lceil n/2 \rceil +2$ and so by Lemma~\ref{lem:big-sets-perc}, $B$ percolates.

Suppose now that $|A_r| = 1$ and let $A_r = \{x\}$.  If $x$ has fewer than $r$ neighbours in $A$, then $A \setminus \{x\}$ is a closed set of size $\frac{n-1}{2}$ and so by Proposition~\ref{prop:floor-half-r}, $G$ contains a set of size $r$ that percolates.  The remainder of the proof involves considering many different sets of size $r$ and showing that if none percolate, then Fact~\ref{fact:big-half-r} can be used to deduce sufficient structural information about $G$ to find a small percolating set.

\begin{center}
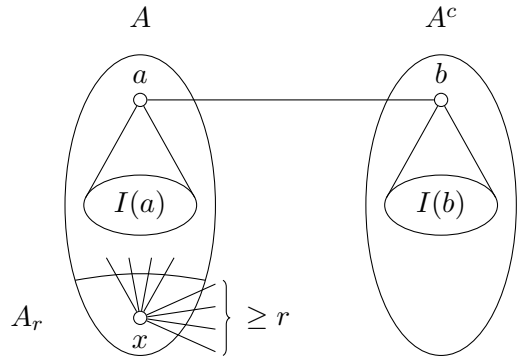
\begin{figure}
\begin{tikzpicture}
	\tikzstyle{vertex}=[circle, draw=black,  minimum size=5pt,inner sep=0pt]
	
	\filldraw[draw=black, fill = white] (0, 0) ellipse (1cm and 2cm);
	\filldraw[draw=black, fill = white] (4, 0) ellipse (1cm and 2cm);
	\draw (-0.866, -1) to[bend left=10] (0.866, -1);
	\node at (-1.5, -1.5) {$A_r$};
	\node[vertex, label=below:{$x$}] at (0, -1.5) (x) {};
	\node at (0, 2.5) {$A$};
	\node at (4, 2.5) {$A^c$};
	
	\node[vertex, label=above:{$a$}] at (0, 1.4) (a) {};
	\node[vertex, label=above:{$b$}] at (4, 1.4) (b) {};
	
	\draw (0, 0) ellipse (0.75cm and 0.4cm);
	\node at (0, 0) {$I(a)$};
	\draw (0.725, 0.1) -- (a) -- (-0.725, 0.1);
	
	\draw (4, 0) ellipse (0.75cm and 0.4cm);
	\node at (4, 0) {$I(b)$};
	\draw (4.725, 0.1) -- (b) -- (3.275, 0.1);
	
	\draw (a) -- (b);
	
	\draw (x) -- ++ (1, 0.45);
	\draw (x) -- ++ (1, 0.15);
	\draw (x) -- ++ (1, -0.15);
	\draw (x) -- ++ (1, -0.45);
	\draw [decorate, decoration={brace}] (1.1, -1) -- (1.1, -2);
	\node at (1.7, -1.5) {$\geq r$}; 
	
	\draw (x) -- ++ (-0.45, 0.8);
	\draw (x) -- ++ (-0.15, 0.8);
	\draw (x) -- ++ (0.15, 0.8);
	\draw (x) -- ++ (0.45, 0.8);

\end{tikzpicture}

\caption{The sets $I(a)$ and $I(b)$ for an edge $\{a, b\}$.}\label{fig:big-half-r-Isets}
\end{figure}
\end{center}

For every vertex $a \in V(G) \setminus \{x\}$, choose a set, denoted $I(a)$ of $r-1$ neighbours of $a$ so that, if $a \in A\setminus\{x\}$, then $I(a) \subseteq A\setminus\{x\}$ and if $a \in A^c$, then $I(a) \subseteq A^c$.  Every vertex in $A$ has at least one neighbour in $A^c$ and since $|A_r| =1$, then every vertex in $A^c$ has at least one neighbour in $A\setminus A_r$.  For any pair $\{a, b\} \in E(G)$ with $a \in A\setminus A_r$ and $b \in A^c$, then 
\begin{align*}
\langle \{b\} \cup I(a) \rangle_r &\supseteq A \cup \{b\}, \text{ and}\\
\langle \{a\} \cup I(b) \rangle_r &\supseteq A^c \cup A_r \cup \{a\}.
\end{align*}
See Figure~\ref{fig:big-half-r-Isets}.  Since each of the sets $\{b\} \cup I(a)$ and $\{a\} \cup I(b)$ each span a set of size at least $\frac{n+1}{2} + 1 = \frac{n-1}{2} +2$, either one of them percolates, or else by Fact~\ref{fact:big-half-r}, for every $a \in A\setminus A_r$, the graph induced by $G$ on $A \setminus \{a, x\}$ is a clique with every vertex having exactly $r-1$ neighbours in $A^c \cup \{a, x\}$ and similarly, for every $b \in A^c$, the set $A^c \setminus \{b\}$ induces a clique with every vertex having $r-1$ neighbours in $A \cup \{b\}$. 

\begin{center}
\begin{figure}
\begin{tikzpicture}
	\tikzstyle{vertex}=[circle, draw=black,  minimum size=5pt,inner sep=0pt]
	
	\filldraw[draw=black, fill = gray!20] (0, 0) ellipse (1cm and 2cm);
	\filldraw[draw=black, fill = gray!20] (4, 0) ellipse (1cm and 2cm);
	\draw (-0.866, -1) to[bend left=10] (0.866, -1);
	\node at (-1, -1.5) {$A_x$};
	\draw (3.134, -1) to[bend left=10] (4.866, -1);
	\node at (5, -1.5) {$B_x$};
	\node at (0, 2.5) {$A\setminus \{x\}$};
	\node at (4, 2.5) {$A^c$};
	
	\node[vertex, label=below:{$x$}] at (2, -2.5) (x) {};
	\draw (-0.1, -1.99) -- (x) --  (0.866, -1);
	\draw (3.134, -1) -- (x) -- (4.1, -1.99);

	\node[vertex, fill = white] at (0, -1.5) (a1) {};
	\node[vertex,  fill = white] at (0, 1.5) (a2) {};
	\node[vertex,  fill = white] at (4, -1.5) (b1) {};
	\node[vertex,  fill = white] at (4, 1.5) (b2) {};
	\draw (a1) -- (x) -- (b1);
	
	\draw (a2) -- ++ (1, 0.3);
	\draw (a2) -- ++ (1, 0);
	\draw (a2) -- ++ (1, -0.3);
	\draw (b2) -- ++ (-1, 0.3);
	\draw (b2) -- ++ (-1, 0);
	\draw (b2) -- ++ (-1, -0.3);
	\draw [decorate,decoration={brace}] (1.1, 1.9) --(1.1, 1.1);
	\draw [decorate,decoration={brace}] (2.9, 1.1) --(2.9, 1.9);
	\node at (2, 1.5) {$r-2$};
	
	\draw (a1) -- ++ (1.3, 1.5);
	\draw (a1) -- ++ (1.3, 1);
	\draw (b1) -- ++ (-1.3, 1.5);
	\draw (b1) -- ++ (-1.3, 1);
	\draw [decorate,decoration={brace}] (1.4, 0.1) --(1.4, -0.55);
	\draw [decorate,decoration={brace}] (2.6, -0.55) --(2.6, 0.1);
	\node at (2, -0.25) {$r-3$};

\end{tikzpicture}
\caption{Structure of the graph if no set $\{a\} \cup I(b)$ or $\{b\} \cup I(a)$ percolates.}\label{fig:r-big-half}
\end{figure}
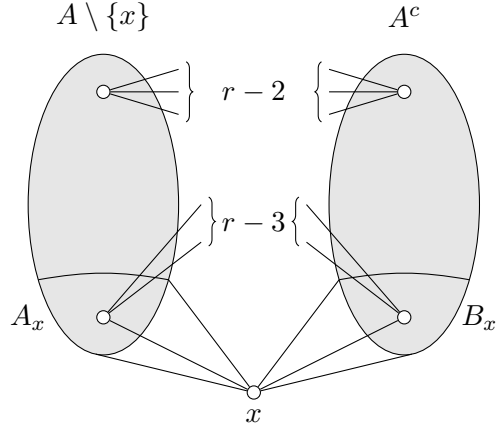
\end{center}

Note that any graph $H$ on at least $3$ vertices with the property that deleting any vertex gives a clique is itself a clique.  Thus, each of $G[A \setminus \{x\}]$ and $G[A^c]$ is a complete graph where every vertex in $A \setminus \{x\}$ has exactly $r-2$ neighbours in $A^c \cup \{x\}$ and every vertex in $A^c$ has exactly $r-2$ neighbours in $A$.

Set $A_x = A \cap N(x)$, $A_1 = A \setminus (\{x\} \cup N(x))$, $B_x = A^c \cap N(x)$ and $B_1 = A^c \setminus N(x)$.  By a previous comment, $|A_x|, |B_x| \geq r$.   Note that every vertex in $A_1$ has $r-2 \geq 2$ neighbours in $A^c$ and every vertex in $A_x$ has $r-3 \geq 0$ neighbours in $A^c$; see Figure~\ref{fig:r-big-half}.

If any vertex $a \in A \setminus \{x\}$ has two neighbours $b_1, b_2 \in B_x$, then let $\{v_1, v_2, \ldots, v_{r-2}\}$ be any $r-2$ vertices in $A_x \setminus \{a\}$ and consider initially infecting $\{b_1, b_2, v_1, v_2, \ldots, v_{r-2}\}$.  Both $x$ and $a$ are adjacent to all infected vertices and so become infected at the first time step.  Thereafter, the remainder of $A_x$ is infected and hence all of $A$.  Since at least $|A \cup \{b_1, b_2\}| = \lceil n/2 \rceil +2$ vertices are infected, the set percolates, by Lemma~\ref{lem:big-sets-perc}.

If any vertex $a \in A \setminus \{x \}$ has two neighbours $b_1, b_2 \in A^c$ with $b_1 \in B_1$, then since $b_1$ has at least $r-2\geq 2$ neighbours in $A \setminus \{x\}$, let $c \in A \setminus \{a, x\}$ be any other neighbour of $b_1$.  Let $v_1, v_2, \ldots, v_{r-2} \in A^c\setminus \{b_1, b_2\}$ be any $r-2$ vertices in $A^c$ and consider initially infecting the set $\{a, c, v_1, v_2, \ldots, v_{r-2}\}$.  In the first step $b_1$ is infected, then $b_2$ and subsequently the remainder of $A^c$ and so also $x$.  As at least $|A^c \cup \{a, c, x\}| = \lceil n/2 \rceil +2$ vertices are infected, the set percolates, by Lemma~\ref{lem:big-sets-perc}.

By symmetry, the same is true for any vertex in $A^c$ with two neighbours in $A_x$ or else two neighbours in $A \setminus \{x\}$, one of which is in $A_1$.

For any $r \geq 5$, every vertex in $A_1 \cup A_x$ has at least $r-3 \geq 2$ neighbours in $A^c$ and so either some vertex has $2$ neighbours in $B_x$ or $2$ neighbours one of which is in $B_1$.  In either case, there is some set of size $r$ that percolates.

\begin{center}
\begin{figure}
\begin{tikzpicture}
	
	\tikzstyle{vertex}=[circle, draw=black,  minimum size=5pt,inner sep=0pt]
	
	\filldraw[draw=black, fill = gray!20] (0, 0) ellipse (1cm and 2cm);
	\filldraw[draw=black, fill = gray!20] (4, 0) ellipse (1cm and 2cm);
	\node at (0, 2.5) {$A_x$};
	\node at (4, 2.5) {$B_x$};
	
	\node[vertex, label=below:{$x$}] at (2, -3) (x) {};
%	\draw (-0.1, -1.99) -- (x) --  (0.866, -1);
%	\draw (3.134, -1) -- (x) -- (4.1, -1.99);
%	
	
	\node[vertex, fill = black, label=left:{$a$}] at (0, 1.5) (a1) {};
	\node[vertex,  fill = black, label=left:{$b$}] at (0, 1) (a2) {};
	\node[vertex, fill = white] at (0, 0.5) (a3) {};
	\node at (0, 0.1) {$\vdots$};
	\node[vertex, fill = white] at (0, -0.5) (a4) {};
	\node[vertex, fill = white] at (0, -1) (a5) {};
	\node[vertex, fill = white] at (0, -1.5) (a6) {};
		
	\node[vertex, fill = white] at (4, 1.5) (b1) {};
	\node[vertex,  fill = white] at (4, 1) (b2) {};
	\node[vertex, fill = white] at (4, 0.5) (b3) {};
	\node at (4, 0.1) {$\vdots$};
	\node[vertex, fill = white] at (4, -0.5) (b4) {};
	\node[vertex, fill = black, label=right:{$c$}] at (4, -1) (b5) {};
	\node[vertex, fill = black, label=right:{$d$}] at (4, -1.5) (b6) {};
	
	\draw (a1) -- (x) -- (a2);
	\draw (a3) -- (x) -- (a4);
	\draw (a5) -- (x) -- (a6);
	\draw (x) -- (0.1, -0.1);
	\draw (b1) -- (x) -- (b2);
	\draw (b3) -- (x) -- (b4);
	\draw (b5) -- (x) -- (b6);
	\draw (x) -- (3.9, -0.1);
	\draw (a1) -- (b1);
	\draw (a2) -- (b2);
	\draw (a3) -- (b3);
	\draw (a4) -- (b4);
	\draw (a5) -- (b5);
	\draw (a6) -- (b6);

\end{tikzpicture}\caption{Case where $r = 4$ and every vertex in $A\setminus \{x\}$ has only one neighbour in $A^c$.}\label{fig:r-big-half-cases}
\end{figure}
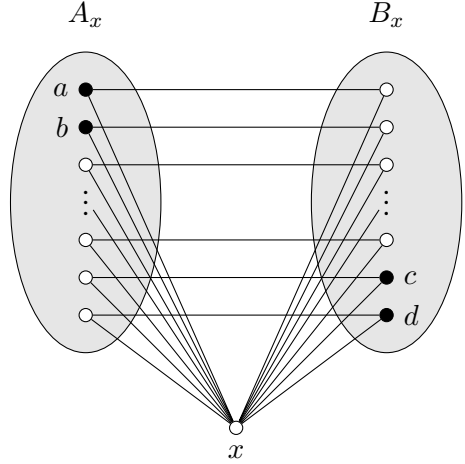
\end{center}

The only remaining case is when $r = 4$ and there are no vertices in $A\setminus \{x\}$ with two neighbours in $A^c$ and similarly, no vertices in $A^c$ with two neighbours in $A \setminus \{x\}$.  That is, $A_1 = B_1 = \emptyset$ and $G$ consists of a clique on $A_x$, a clique on $B_x$, all vertices in $A_x \cup B_x$ joined to $x$ and a perfect matching between $A_x$ and $B_x$, as in Figure~\ref{fig:r-big-half-cases}.  Since $(n-1)/2 \geq 4$, choose $a, b \in A_x$ and $c, d \in B_x$ with $c, d \notin N(a) \cup N(b)$ and initially infect the set $\{a, b, c, d\}$.  The vertex $x$ is infected at the first time step.  At the second time step, the neighbours of $a$ and $b$ in $A^c$ and the neighbours of $c$ and $d$ in $A_x$ are infected and then all remaining vertices are infected in the third time step.

This complete the proof in the case that $\langle A \rangle_r = A$ and $|A| = \left\lceil \frac{n}{2} \right\rceil$.
\end{proof}

The final remaining case to be dealt with is the following.

\begin{proposition}\label{prop:vbig-half-r}
Let $r \geq 4$, $n$ be sufficiently large, and let $G$ be a graph on $n$ vertices with $\delta(G) \geq \lfloor n/2 \rfloor + (r-3)$.  If there is a set $A \subseteq V(G)$ with $|A| = \lceil n/2 \rceil +1$ and $\langle A \rangle_r = A$, then $m(G, r) = r$.
\end{proposition}

\begin{proof}
By Fact~\ref{fact:big-half-r}, the set $A^c$ induces a complete graph with every vertex having exactly $r-1$ neighbours in $A$.  Any vertex $x \in A$ has
\[
\deg_{A^c}(x) \geq \lfloor n/2 \rfloor + (r-3) - \lceil n/2 \rceil =
\begin{cases}
	r-3	&\text{if $n$ is even},\\
	r-4	&\text{if $n$ is odd}
\end{cases}.
\]
As in previous proofs, set $A_r = \{x \in A \mid \deg_{A^c}(x) \geq r\}$.  Again, if $A_r = \emptyset$, then $A^c$ is a closed set of size $\lfloor n/2 \rfloor -1$ and so by Proposition~\ref{prop:vsmall-half-r}, there is a percolating set of size $r$.  Thus, assume that $A_r \neq \emptyset$.  Every vertex in $A\setminus A_r$ has at most $3$ non-neighbours in $A$.  If $A_r \nsubseteq \langle A \setminus A_r \rangle_r$, then there is a closed set that is smaller than $A$ and so by one of Propositions~\ref{prop:vsmall-half-r}, \ref{prop:floor-half-r}, or \ref{prop:ceil-half-r}, $G$ has a percolating set of size $r$.  Therefore, assume that $\langle A \setminus A_r \rangle_r = A$.

Note that since
\[
(r-1)\left( \lfloor n/2 \rfloor -1 \right) = e(A, A^c) \geq r|A_r|,
\]
then $|A_r| \leq \frac{(r-1)}{r}(\lfloor n/2 \rfloor -1) \leq \lceil n/2 \rceil - (r+3)$ as long as $n \geq 2(r^2+2r+1)$.

If there is any vertex $a \in A \setminus A_r$ with a neighbour $b \in A^c$, then since $a$ has at most $3$ non-neighbours in $A$, there are at least $r-1$ neighbours of $a$ in $A \setminus A_r$.  Let $v_1, v_2, \ldots, v_{r-1}$ be any neighbours of $a$ in $A \setminus A_r$.  Since the set $\{b, v_1, v_2, \ldots, v_{r-1}\}$ infects $a$ and hence all of $A\setminus A_r$ and subsequently $A_r$, the closure of this set has at least $|A| + 1 = \lceil n/2 \rceil +2$ vertices and hence the set percolates by Lemma~\ref{lem:big-sets-perc}.

The only case in which there can be no edges between $A \setminus A_r$ and $A^c$ is when $r = 4$, $n$ is odd, every vertex in $A^c$ has $r-1 = 3$ neighbours in $A_r$ and $G[A \setminus A_r]$ is a complete graph with all vertices in $A \setminus A_r$ adjacent to every vertex in $A_r$.  In this case $|A_r| \geq 3$.   If $|A_r| \geq 4$, then any set of size $4$ in $A^c$ percolates.  Therefore, assume that $|A_r| = 3$.  The graph is as in Figure~\ref{fig:vbig-half-r4}. Then any set consisting of two vertices from $A \setminus A_r$ and two vertices from $A^c$ will infect all of $A_r$ and subsequently the remainder of the graph. 

\begin{center}
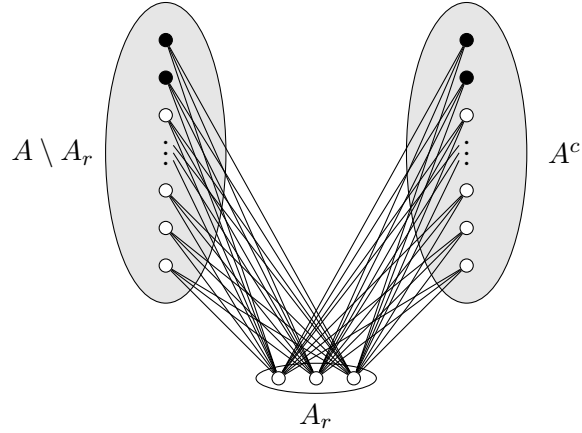
\begin{figure}[htb]
\begin{tikzpicture}
	
	\tikzstyle{vertex}=[circle, draw=black,  minimum size=5pt,inner sep=0pt]
	
	\filldraw[draw=black, fill = gray!20] (0, 0) ellipse (0.8cm and 2cm);
	\filldraw[draw=black, fill = gray!20] (4, 0) ellipse (0.8cm and 2cm);
	\node at (-1.5, 0) {$A\setminus A_r$};
	\node at (5.3, 0) {$A^c$};
	
	\draw (2, -3) ellipse (0.8cm and 0.2cm);
	\node[vertex] at (1.5, -3) (c1) {};
	\node[vertex] at (2, -3) (c2) {};
	\node[vertex] at (2.5, -3) (c3) {};
	\node at (2, -3.5) {$A_r$};
		
	\node[vertex, fill = black] at (0, 1.5) (a1) {};
	\node[vertex,  fill = black] at (0, 1) (a2) {};
	\node[vertex, fill = white] at (0, 0.5) (a3) {};
	\node at (0, 0.1) {$\vdots$};
	\node[vertex, fill = white] at (0, -0.5) (a4) {};
	\node[vertex, fill = white] at (0, -1) (a5) {};
	\node[vertex, fill = white] at (0, -1.5) (a6) {};
		
	\node[vertex, fill = black] at (4, 1.5) (b1) {};
	\node[vertex,  fill = black] at (4, 1) (b2) {};
	\node[vertex, fill = white] at (4, 0.5) (b3) {};
	\node at (4, 0.1) {$\vdots$};
	\node[vertex, fill = white] at (4, -0.5) (b4) {};
	\node[vertex, fill = white] at (4, -1) (b5) {};
	\node[vertex, fill = white] at (4, -1.5) (b6) {};
	
	\draw (a1) -- (c1) -- (a2);
	\draw (a3) -- (c1) -- (a4);
	\draw (a5) -- (c1) -- (a6);
	\draw (c1) -- (0.1, -0.1);
	\draw (b1) -- (c1) -- (b2);
	\draw (b3) -- (c1) -- (b4);
	\draw (b5) -- (c1) -- (b6);
	\draw (c1) -- (3.9, 0.1);

	\draw (a1) -- (c2) -- (a2);
	\draw (a3) -- (c2) -- (a4);
	\draw (a5) -- (c2) -- (a6);
	\draw (c2) -- (0.1, 0);
	\draw (b1) -- (c2) -- (b2);
	\draw (b3) -- (c2) -- (b4);
	\draw (b5) -- (c2) -- (b6);
	\draw (c2) -- (3.9, 0);	
	
	\draw (a1) -- (c3) -- (a2);
	\draw (a3) -- (c3) -- (a4);
	\draw (a5) -- (c3) -- (a6);
	\draw (c3) -- (0.1, 0.1);
	\draw (b1) -- (c3) -- (b2);
	\draw (b3) -- (c3) -- (b4);
	\draw (b5) -- (c3) -- (b6);
	\draw (c3) -- (3.9, -0.1);	
\end{tikzpicture}

\caption{$r = 4$ and no edges between $A \setminus A_r$ and $A^c$.}\label{fig:vbig-half-r4}
\end{figure}
\end{center}

In all cases, there is some set of $r$ vertices that percolates and so $m(G, r) = r$.
\end{proof}

The proof of Theorem~\ref{thm:main-ub} can now be completed.

\begin{proof}[Proof of Theorem~\ref{thm:main-ub}]
Let $n$ be large enough to apply the lemmas and propositions given previously and let $G$ be a graph on $n$ vertices with $\delta(G) \geq \lfloor n/2 \rfloor + (r-3)$.  By Lemma~\ref{lem:complete-bip-r}, $G$ contains a copy of $K_{r, r-1}$ and the $r$ vertices in one partition set, $A$, have closure $|\langle A \rangle_r| \geq 2r-1 > 2(r-1)$.  By Lemma~\ref{lem:big-sets-perc} and Proposition~\ref{prop:no-mid-size-closed}, either $A$ percolates or else $|\langle A \rangle_r| \in \left[\lfloor n/2 \rfloor -1, \lceil n/2 \rceil + 1\right]$. If $|\langle A \rangle_r| \in \left[\lfloor n/2 \rfloor -1, \lceil n/2 \rceil + 1\right]$, then by Proposition~\ref{prop:vsmall-half-r}, \ref{prop:floor-half-r}, \ref{prop:ceil-half-r}, or \ref{prop:vbig-half-r}, $G$ contains a percolating set of size $r$.
\end{proof}

\section{Open problems}\label{sec:open}

There are a number of natural questions related to the results in this paper that remain open.  One could ask for the conditions on $\delta(G)$ that guarantee $m(G, r) \leq k$ for a fixed $k \geq r+1$.  Following the line of inquiry in \cite{DFLMPU16} and \cite{FPR15}, one might consider the lower bounds on $\sigma_2(G)$ that guarantee that $m(G, r) = r$ for $r \geq 3$.  A problem that may be quite technical would be the characterization of those small graphs for which $\delta(G) = \lfloor n/2 \rfloor + \min\{1, r-3\}$ but $m(G, r) > r$.

\end{document}